\documentclass[12pt]{amsart}
\usepackage{amssymb,latexsym,amsmath,amsthm,amscd}
\usepackage{setspace}
\usepackage[active]{srcltx}
\usepackage[all]{xy} \xyoption{arc}
\usepackage[left=3cm,top=2cm,right=3cm,bottom = 2cm]{geometry}
\usepackage{graphicx}
\usepackage{tcolorbox}
\usepackage{hyperref}

\theoremstyle{plain}
\newtheorem{thm}{Theorem}[section]
\newtheorem{lem}[thm]{Lemma}
\newtheorem{prop}[thm]{Proposition}

\theoremstyle{definition}
\newtheorem{dfn}[thm]{Definition}
\newtheorem{ex}[thm]{Example}

\theoremstyle{remark}
\newtheorem{rmk}[thm]{Remark}

\newcommand{\cO}{\mathcal{O}}

\newcommand{\veps}{\varepsilon}

\DeclareMathOperator{\uhp}{\mathcal{H}}

\DeclareMathOperator{\Tr}{Tr}

\DeclareMathOperator{\Hom}{Hom}

\DeclareMathOperator{\GL}{GL}

\DeclareMathOperator{\SL}{SL}

\DeclareMathOperator{\Res}{Res}

\DeclareMathOperator{\diag}{diag}

\newcommand*{\df}{\mathrel{\vcenter{\baselineskip0.5ex \lineskiplimit0pt
                     \hbox{\scriptsize.}\hbox{\scriptsize.}}} =}

\providecommand{\norm}[1]{\left\lVert#1\right\rVert}
\providecommand{\abs}[1]{\left\lvert#1\right\rvert}

\providecommand{\twomat}[4]{\left(\begin{matrix}#1&#2\\#3&#4\end{matrix}\right)}
\providecommand{\stwomat}[4]{\left(\begin{smallmatrix}#1&#2\\#3&#4\end{smallmatrix}\right)}

\newcommand{\CC}{\mathbf{C}}

\newcommand{\ZZ}{\mathbf{Z}}

\newcommand{\RR}{\mathbf{R}}

\renewcommand{\Im}{\operatorname{Im}}
\renewcommand{\Re}{\operatorname{Re}}
\DeclareMathOperator{\Herm}{Herm}
\DeclareMathOperator{\Pos}{Pos}

\DeclareMathOperator{\be}{\mathbf{e}}
\DeclareMathOperator{\Kl}{Kl}

\begin{document}
\title{Eisenstein metrics}
\author{Cameron Franc}
\address{McMaster University}
\email{franc@math.mcmaster.ca}
\thanks{The author gratefully acknowledges financial support received from NSERC through a Discovery Grant.}

\date{}

\begin{abstract}
  We study families of metrics on automorphic vector bundles associated to representations of the modular group. These metrics are defined using an Eisenstein series construction. We show that in certain cases, the residue of these Eisenstein metrics at their rightmost pole is a harmonic metric for the underlying representation of the modular group.  The last section of the paper considers the case of a family of representations that are indecomposable but not irreducible. The analysis of the corresponding Eisenstein metrics, and the location of their rightmost pole, is an open question whose resolution depends on the asymptotics of matrix-valued Kloosterman sums.
\end{abstract}
\maketitle
\tableofcontents

\section{Introduction}
This paper concerns results on functions on the complex upper-half plane that take values in hermitian matrices and which satisfy automorphic transformation laws. These functions are uniformizations of metrics on automorphic vector bundles, and the motivation for their study stems from modern Hodge theory. Nonabelian Hodge correspondences describe equivalences between categories of stable connections and categories of Higgs bundles on an underlying base manifold. Such highly nontrivial correspondences have found use throughout geometry, topology, physics and even in number theory, beginning primarily with Ngo's proof of the fundamental lemma \cite{Ngo} using properties of the Hitchin integrable system \cite{HitchinDuke}, \cite{Hitchin} associated with moduli spaces of Higgs bundles. In this paper we introduce hermitian matrix valued Eisenstein series and study to what extent harmonic metrics realizing the nonabelian Hodge correspondence could be described as residues of such Eisenstein series. This program is carried out fully for the two-dimensional inclusion representation of the modular group, and some general results and difficulties are studied when the monodromy at the cusp is unitary. Before getting to the details we shall provide background and motivation, as well as a summary of our results.

The nonabelian Hodge correspondence traces back to work of Narasimhan-Seshadri \cite{NarasimhanSeshadri} in the compact case, and Mehta-Seshadri \cite{MehtaSeshadri} in the noncompact setting. These works established correspondences between categories of unitary representations of fundamental groups and categories of holomorphic vector-bundles on a base curve. Later authors, beginning with work of Hitchin \cite{HitchinDuke}, \cite{Hitchin}, Donaldson \cite{Donaldson}, Corlette \cite{Corlette} and Simpson \cite{Simpson1}, \cite{Simpson2}, extended this to encompass all irreducible representations of the fundamental group, by enhancing vector bundles $E/X$ with the additional structure of a \emph{Higgs field}, which is an $\cO_X$-linear map
\[
  \theta \colon E \to E\otimes \Omega^1_X
\]
satisfying $\theta^2=0$, a condition which is automatic for curves.

One reason why nonabelian Hodge correspondences have proven so useful is that they can be used to replace nonlinear objects --- regular connections --- with $\cO_X$-linear Higgs fields. For example, recently in joint work with Steven Rayan, we used this strategy in \cite{FrancRayan} to establish new instances of inequalities governing the multiplicities among the line bundles that arise in decompositions of vector-bundles associated with vector-valued modular forms. In past joint work with Geoffrey Mason \cite{FrancMason1},  we established instances of such inequalities by proving the existence of semi-canonical forms for the modular derivative $D_k = q\frac{d}{dq} -\tfrac{k}{12}E_2$ acting on spaces of modular forms of weight $k$. The nonlinear nature of these operators proves to be a nontrivial obstacle in such arguments.

Unfortunately, at the heart of establishing nonabelian Hodge correspondences lies the problem of solving a nonlinear partial differential equation, the solution of which yields the existence of \emph{harmonic metrics} (cf. Definition \ref{d:harmonic}) on vector-bundles that can be used to associate Higgs bundles to regular connections, and vice-versa. For an overview of how such correspondences work, see \cite{FrancRayan}, \cite{RabosoRayan}, or \cite{GoldmanXia} for the rank-one case. Existence proofs for harmonic metrics can be executed using a heat-flow argument as in \cite{Donaldson}, but writing down explicit examples of harmonic metrics can be difficult except in special circumstances.

If the base manifold $X$ is the compactification of some quotient $\Gamma \backslash \uhp$ where $\uhp$ is the complex upper-half plane, and $\Gamma$ is a Fuchsian group, then vector-bundles on $X$ are pulled back to trivializable bundles on $\uhp$. Attendant structures on such vector-bundles, such as connections, Higgs fields, or metrics, can then be described as automorphic objects on $\uhp$, typically vector or matrix-valued, satisfying a prescribed transformation law under $\Gamma$. In this paper we explore the use of Eisenstein series for constructing metrics on automorphic vector-bundles, focusing on the case of $\Gamma = \SL_2(\ZZ)$, so that $X = Y \cup \{\infty\}$ where $Y = \Gamma \backslash \uhp$ is the moduli space of elliptic curves. We are primarily interested in representations that are not unitary, so that the corresponding harmonic metric is different from the Petersson inner-product. See \cite{Deitmar}, \cite{DeitmarMonheim1}, \cite{DeitmarMonheim2}, \cite{Muller} for recent work on analysis of automorphic forms transforming under non-unitary representations.

In Section \ref{s:eisenstein} we associate \emph{Eisenstein metrics} $H(\tau,s)$ to representations of $\Gamma$ generalizing constructions of \cite{KnoppMason2}, \cite{KnoppMason3}, \cite{KnoppMason4}, and prove their convergence when the real part of $s$ is large enough. We state our main convergence result, Proposition \ref{p:convergence}, in a form that is flexible enough to work for complex analytic families of representatations of $\Gamma$. Such families of Eisenstein series are examples of higher dimensional analogues of families studied in \cite{Bruggeman}. Following this, Section \ref{s:harmonic} then briefly recalls the definition of a harmonic metric.

In the standard theory of Eisenstein series one starts with a simple function satisfying a linear differential equation and averages to get a more interesting function satisfying a larger set of invariance properties. By linearity, the averaged function satisfies the same linear differential equation. If one instead hopes to solve a nonlinear differential equation, such averaging cannot be expected to solve nonlinear equations in the same way that one solves linear equations. As such, the following result may be somewhat surprising:

\begin{thm}
  \label{t:main}
  Let $\rho$ be the inclusion representation of $\SL_2(\ZZ)$, and let $H(\tau,s)$ be the corresponding Eisenstein metric. Then $H(\tau,s)$ admits analytic continuation to $\Re(s) > \frac 32$ with a simple pole at $s=2$ of residue equal to a tame harmonic metric for $\rho$.
\end{thm}

See Theorem \ref{t:incmain} for a more precise statement of this result, which is proved by computing the Fourier coefficients of $H(\tau,s)$ using standard techniques from the theory of Eisenstein series.

In Section \ref{s:unitary} we consider $H(\tau,s)$ for representations where $\rho(T)$ is unitary, where $T =\stwomat 1101$ is the cuspidal stabilizer, and we work out an expression for the Fourier coefficients of $H(\tau,s)$. When $\rho$ is itself unitary, this expression shows that $H(\tau,s)$ has a simple pole at $s=1$ of residue equal to a scalar multiple of the identity matrix, which is the Petersson inner-product for unitary representations. Thus, the analogue of Theorem \ref{t:main} is true when $\rho(T)$ is unitary, except that the pole shifts left to $s=1$. The difference between the unitary case and the case of Theorem \ref{t:main} seems to be the nontrivial $(2\times 2)$-Jordan block in $\rho(T)$ from Theorem \ref{t:main}, whereas $\rho(T)$ is diagonalizable for unitary $\rho$.

It is unclear to this author to what extent one might expect to recover harmonic metrics as residues of Eisenstein metrics, and so to probe this question in Section \ref{s:indecomposable} we consider a family of non-unitarizable representations $\rho$ such that $\rho(T)$ is of finite order (hence unitarizable). Unfortunately, the difficulty in using the Fourier coefficient computations of Section \ref{s:unitary} in general rests in understanding some matrix-valued nonabelian Kloosterman sums and associated generating series. To describe these sums, if $\rho$ is a representation of $\Gamma$, $L$ is a matrix satisfying $\rho(T) = \be(L) \df e^{2\pi i L}$, and $h$ is a Hermitian positive-definite matrix satisfying
\begin{equation}
\label{eq:hinv}
  \rho(\pm T)^th\overline{\rho(\pm T)} = h,
\end{equation}
then the associated \emph{Kloosterman sums} are defined as
\[
  \Kl(\rho,L, c) = \sum_{\substack{d=1\\\gcd(c,d)=1}}^c \be(-L\tfrac dc)\rho\stwomat abcd^th \overline{\rho\stwomat abcd} \be(L\tfrac dc),
\]
where $a,b \in \ZZ$ are chosen so that $ad-bc=1$. The invariance property of equation \eqref{eq:hinv} ensures that $\Kl(\rho,L,c)$ is well-defined independent of this choice, and the exponentials in the definition of $\Kl(\rho,L,c)$ ensure that the summands only depend on the value of $d$ modulo $c$.

As is typical in the theory of Eisenstein series, understanding the analytic continuation of Eisenstein metrics $H(\tau,s)$ rests in coming to grips with the analytic properties of Kloosterman sum generating series
\[
  D(s) = \sum_{c\geq 1} \frac{\Kl(\rho,L,c)}{c^s}.
\]

In Section \ref{s:indecomposable} we analyze these sums for a family of representations and exponents $(\rho,L)$ arising from a group cocycle obtained by integrating $\eta^4$, where $\eta$ is the Dedekind eta function. This family of two-dimensional representations, studied in \cite{MarksMason}, contains a one-dimensional subrepresentation that does not split off as a direct summand for generic specializations of the family. We show in Proposition \ref{p:taylor} that the family of Kloosterman sums $\Kl(\rho,L,c)$ admit a sort of second-order Taylor expansion in the deformation parameters about the specialization where the family becomes decomposable. The second-order term in this Taylor expansion contains a sequence $a_c$ of positive integers, whose values are in Table \ref{t:values} on page \pageref{t:values} and plotted in Figure \ref{f:plots} on page \pageref{f:plots}. Determining the rightmost pole of $H(\tau,s)$ for this family of representations comes down to understanding the growth of this sequence $a_c$. For example, if one could prove that $a_c = O(\phi(c)\log(c))$, where $\phi$ is the Euler phi function, then the rightmost pole of $H(\tau,s)$ would occur at $s=1$ and the corresponding residue would be positive-definite. If instead $a_c = O(\phi(c) c^\veps)$ for some $\veps > 0$, then the rightmost pole of $H(\tau,s)$ would occur to the right of $s=1$ and the residue would not be positive definite, hence not a harmonic metric. As we have only computed the terms $a_c$ for $c\leq 5000$, we are not prepared to make a conjecture as to the expected growth of sequences such as $a_c$. A natural question is to ask whether the $\ell$-adic machinery developed in \cite{Katz} and subsequent work could be employed to study such asymptotic questions, and we plan to return to this in future work.

\begin{rmk}
  The Kloosterman sums above are special cases of more general matrix-valued Kloosterman sums
  \[
  \sum_{\substack{d=1\\ \gcd(c,d)=1}}^c \be(-L\tfrac ac)\rho\stwomat abcd\be(-L\tfrac dc)
\]
that appeared in \cite{KnoppMason2}, \cite{KnoppMason3}, \cite{KnoppMason4}. Notice that if $\rho$ is trivial then this yields classical Kloosterman sums. In the definition of $\Kl(\rho,L,c)$ above, the representation $\rho$ is replaced by its induced action on the space $\Herm_d$ of $(d\times d)$-Hermitian matrices.
\end{rmk}

\subsection{Notation and conventions}
\begin{enumerate}
\item[---] $\Gamma = \SL_2(\ZZ)$.
\item[---] $T =\stwomat 1101$, $S = \stwomat{0}{-1}{1}0$.
\item[---] $\uhp = \{x+iy\in\CC \mid y > 0\}$ is the complex upper half plane.
\item[---] $\be(M) = e^{2\pi i M}$ for complex matrices $M$.
\item[---] In this note all metrics are Hermitian and $\Herm_d$ denotes the space of $(d\times d)$ Hermitian matrices.
\item[---] Representations are complex and finite-dimensional.
\end{enumerate}

\section{Eisenstein metrics}
\label{s:eisenstein}
Let $\rho\colon \Gamma \to \GL_d(\CC)$ be a representation of $\Gamma = \SL_2(\ZZ)$. Let $\Herm_d$ denote the real vector-space of $d\times d$ Hermitian matrices, so that $M\in \Herm_d$ means that $\bar M = M^t$.

\begin{dfn}
  A \emph{metric} for $\rho$ is a smooth function $H\colon \uhp \to \Herm_d$ such that $H(\tau)$ is positive definite for all $\tau \in \uhp$, and such that
\begin{equation}
  \label{eq:trans1}
\rho(\gamma)^t  H(\gamma \tau)\overline{\rho(\gamma)} = H(\tau)
\end{equation}
holds for all $\gamma \in \Gamma$.
\end{dfn}

Such functions can be used to define analogues of Petersson inner products on spaces of vector-valued modular forms associated to $\rho$. For example, if $F$ satisfies $F(\gamma\tau)=\rho(\gamma)F(\tau)$ for all $\gamma \in\Gamma$, and similarly for $G$, then we can define an invariant scalar-valued form by the rule
\[
  \langle F,G\rangle_\tau \df F(\tau)^tH(\tau)\overline{G(\tau)}.
\]
Equation \eqref{eq:trans1} implies that $\langle F,G\rangle_{\gamma \tau} = \langle F,G\rangle_{\tau}$ for all $\gamma \in \Gamma$.

\begin{ex}
If $\rho$ is unitary then $H(\tau) = I_d$ defines the usual Petersson metric. More generally, if $M \in \Herm_d$ is a constant positive definite matrix that satisfies $M\cdot \gamma = M$ with respect to the right action $M\cdot \gamma = \rho(\gamma)^tM\overline{\rho(\gamma)}$ of $\Gamma$, then $H(\tau) = M$ is a metric for $\rho$.
\end{ex}

Let $h\colon \uhp \to \Herm_d$ be a positive definite and smooth function that satisfies equation \eqref{eq:trans1} for $\gamma=\pm T$. More precisely, we assume that
\begin{align}
\label{eq:trans2} \rho(T)^{t}h(\tau+1)\overline{\rho(T)} =& h(\tau),\\
  \label{eq:trans3} \rho(-I)^{t}h(\tau)\overline{\rho(-I)} =& h(\tau).
\end{align}
If $\rho(-I)$ is a scalar matrix, then we must have $\rho(-I) = \pm I_d$, and equation \eqref{eq:trans3} is satisfied. This occurs for example if $\rho$ is irreducible.

Given an $h$ as in the preceding paragraph, the corresponding Poincare series is defined as usual by the formula
\begin{equation}
  \label{eq:eisenstein1}
  P(\rho,h,\tau) \df \sum_{\gamma \in \langle \pm T\rangle\backslash \Gamma} \rho(\gamma)^{t}h(\gamma \tau)\overline{\rho(\gamma)}. 
\end{equation}
This is well-defined thanks to equations \eqref{eq:trans2} and \eqref{eq:trans3}. Furthermore, if $P$ converges absolutely then we have
\begin{align*}
  P(\rho,h,\alpha \tau) =&\sum_{\gamma \in \langle \pm T\rangle\backslash \Gamma} \rho(\gamma)^{t}h(\gamma \alpha\tau)\overline{\rho(\gamma)}\\
  =&\sum_{\gamma \in \langle \pm T\rangle\backslash \Gamma} \rho(\gamma\alpha^{-1})^{t}h(\gamma\tau)\overline{\rho(\gamma\alpha^{-1})}\\
  =& \rho(\alpha)^{-t}P(h,\tau)\overline{\rho(\alpha)}^{-1}
\end{align*}
for all $\alpha \in \Gamma$. Thus, after moving the $\rho$-terms to the other side, one sees that $P(h,\tau)$ satisfies equation \eqref{eq:trans1} as a function of $\tau$. If $h$ is chosen so that $P(\rho,h,\tau)$ converges absolutely to a smooth function, then the series $P(\rho,h,\tau)$ defines a metric for $\rho$.

Let $g \in \GL_d(\CC)$. Notice that if $h$ satisfies equations \eqref{eq:trans2} and \eqref{eq:trans3} for $\rho$, then $g^{-t}h\overline{g}^{-1}$ satisfies equations \eqref{eq:trans2} and \eqref{eq:trans3} for $g\rho g^{-1}$. 
\begin{lem}
  \label{l:pconj}
  Assuming that both Poincare series converge absolutely, then one has
  \[
    P(g\rho g^{-1}, g^{-t}h\bar g^{-1},\tau) = g^{-t}P(\rho,h,\tau)\bar g^{-1}.
  \]
\end{lem}
\begin{proof}
  The proof is a direct computation:
  \begin{align*}
P(g\rho g^{-1}, g^{-t}h\bar g^{-1},\tau) =&\sum_{\gamma \in \langle \pm T\rangle\backslash \Gamma} (g\rho(\gamma)g^{-1})^{t}g^{-t}h(\gamma \alpha\tau)\bar g^{-1}\bar g\overline{\rho(\gamma)}\bar g^{-1}\\
    =& g^{-t}P(\rho,h,\tau)\bar g^{-1}.
  \end{align*}
\end{proof}

Our next goal is to describe convenient choices of functions $h$ satisfying equations \eqref{eq:trans2} and \eqref{eq:trans3}. To this end we introduce the notion of exponents. Define $\be(M) = e^{2\pi i M}$ for matrices $M$.
\begin{dfn}
  \label{d:exponents}
  A choice of \emph{exponents} for $\rho$ is a matrix $L$ such that $\rho(T) = \be(L)$.
\end{dfn}
Since the matrix exponential is surjective, exponents always exist. They can be described explicitly in terms of a Jordan canonical form for $\rho(T)$, cf.  Theorem 3.7 of \cite{CandeloriFranc}, and this description shows that if $X$ commutes with $\rho(T)$, then $X$ also commutes with a choice of exponents.

Let $L$ be a choice of exponents for $\rho$. Since the matrix exponential satisfies $\be(X+Y) = \be(X)\be(Y)$ provided that $XY=YX$, it follows that
\[\be(L(\tau+1)) = \rho(T)\be(L\tau)=\be(L\tau)\rho(T).\]
Therefore, for any $h\in\Herm_d$, the function
\begin{equation}
  \label{eq:hchoice}
  h(\tau) = \be(-L\tau)^th\overline{\be(-L\tau)}
\end{equation}
is Hermitian and satisfies equation \eqref{eq:trans2}. However, $h(\tau)$ need not satisfy equation \eqref{eq:trans3}, and it need not be positive definite in general. Therefore we introduce a set of admissible choices for $h$:
\begin{dfn}
  \label{d:pos}
  Given a representation $\rho\colon \Gamma \to \GL_d(\CC)$ define
  \[
  \Pos(\rho) \df \{h\in \Herm_d \mid h \textrm{ is positive definite and } \rho(-I)^th\overline{\rho(-I)}=h\}.
  \]
\end{dfn}
\begin{rmk}
In many cases the condition that $\rho(-I)^th\overline{\rho(-I)}=h$ in Definition \ref{d:pos} holds automatically for all $h \in\Herm_n$. This is so for example if $\rho$ is irreducible, for then one has $\rho(-I) = \pm I$. In such cases $\Pos(\rho)$ is the set of all positive definite matrices in $\Herm_d$.
\end{rmk}

\begin{prop}
  Suppose that $\rho(-I)$ is unitary. Then the set $\Pos(\rho)$ contains $I$, and it is closed under addition and under multiplication by positive real numbers. Furthermore, if $h\in \Pos(\rho)$, then $g^th\bar g$ is positive definite for all $g \in \GL_d(\CC)$.
\end{prop}
\begin{proof}
  Necessarily $\rho(T)^t\overline{\rho(T)}$ is Hermitian, and it is also necessarily positive definite. The condition that $\rho(-I)$ is unitary says exactly that $I$ satisfies the second definining condition of $\Pos(\rho)$, so that $I \in \Pos(\rho)$ when $\rho(-I)$ is unitary. Closure of $\Pos(\rho)$ under addition and positive real rescalings is clear. Similarly, if $z$ is a complex column vector then $z^tg^th\overline{gz} = w^th\bar w$ where $w = gz$. Since $g$ is invertible, $w$ is never zero, and so $w^th\bar w > 0$. This concludes the proof.
\end{proof}

\begin{dfn}
  \label{d:eisenstein}
  Let $\rho$ denote a representation of $\Gamma$, and let $L$ denote a choice of exponents for $\rho$. Then the associated \emph{Eisenstein metric} is the infinite series
  \[
H(\rho,L,h,\tau,s) \df \sum_{\gamma \in \langle \pm T\rangle\backslash \Gamma} \rho(\gamma)^{t}\be(-L^t(\gamma \cdot \tau))h\overline{\be(-L(\gamma \cdot \tau))}\overline{\rho(\gamma)}\Im(\gamma\cdot\tau)^s,
\]
where $h \in \Pos(\rho)$ and $\tau\in \uhp$.
\end{dfn}
In the definition above, the dot in $\gamma \cdot\tau$ denotes the action of $\Gamma$ on $\uhp$, not matrix multiplication. All other products in the expression defining Eisenstein metrics are ordinary matrix products. When $\rho$ is trivial, $L=0$, and $h=1$, then this definition gives the usual Eisenstein series.

The formation of $H(\rho,L,h,\tau,s)$ is linear in $h$. We will often write $H(h,\tau,s)$ for these Eisenstein metrics when the dependence on $\rho$ and $L$ is clear. Our next goal is to prove that $H(h,\tau,s)$ converges absolutely if the real part of $s$ is large enough. The proof is analogous to the proof of Theorem 3.2 in \cite{KnoppMason2}. We will show that one has convergence when $\Re(s)$ is large even if $(\rho,L)$ varies in a family, as in the following definitions:

\begin{dfn}
  Let $U\subseteq \CC^m$ denote an open subset. A \emph{family of representations} for $\Gamma$ varying analytically over $U$ consists of an analytic map
  \[
  \rho \colon U \to \Hom(\Gamma,\GL_n(\CC)).
  \]
\end{dfn}
If $\rho$ is a family of representations then we usually write $\rho_u$ for the value of $\rho$ at $u \in U$. For each $\gamma \in \Gamma$ we obtain a function $\rho(\gamma) \colon U \to \GL_n(\CC)$, and the analyticity of $\rho$ consists of the analyticity of these maps. It suffices to test analyticity on a set of generators for $\Gamma$. If $\cO(U)$ denotes the ring of analytic functions on $U$, then a family of representations on $U$ is the same thing as a homomorphism
\[
  \rho \colon \Gamma \to \GL_n(\cO(U)).
\]
\begin{dfn}
  A \emph{choice of exponents} for a family of representations $\rho$ on $U\subseteq \CC^m$ consists of a holomorphic map $L \colon U \to M_m(\CC)$ such that $\rho(T) = e^{2\pi i L}$.
\end{dfn}
As above, we will sometimes write $L_u$ for the exponents evaluated at a point $u \in U$.

\begin{ex}
There exists an analytic family of representations on $\CC^\times$ determined by
  \begin{align*}
    \rho(T) &= \twomat uu0{u^{-1}}, & \rho(S) &= \twomat{0}{-u}{u^{-1}}0.
  \end{align*}
  The specializations $\rho_u$ are irreducible as long as $u^4-u^2+1 \neq 0$. Since $\Tr(\rho(T)) = u+u^{-1}$, one sees that this family is nontrivial, in the sense that not all fibers are isomorphic representations, though one does have $\rho_u \cong \rho_{u^{-1}}$. Modulo this identification, this describes one component of the universal family of irreducible representations of $\Gamma$ of rank two, cf. \cite{Mason1}, \cite{TubaWenzl}.

  Let $\log$ denote the branch of the complex logarithm such that the imaginary parts of $\log(z)$ are contained in $[0,2\pi)$. Then a choice of exponents defined on $\CC \setminus \RR_{\geq 0}$ is
  \[
  L = \frac{1}{2\pi i}\twomat{\log(u)}{\frac{(\log(u)-\log(u^{-1}))u^2}{u^2-1}}{0}{\log(u^{-1})}
\]
The apparent singularity at $u=-1$ is removable, while the singularity on the branch cut at $u=1$ is not. These values of $u$ correspond to the specializations of $\rho$ where $\rho(T)$ is not diagonalizable.
 \end{ex}

\begin{dfn}
Let $(\rho,L)$ denote a family of representations and a choice of exponents on some open subset $U\subseteq \CC$. Then an analytic function $P\colon U \to \GL_d(\CC)$ is said to \emph{put $L$ into Jordan canonical form} provided that $PLP^{-1}$ is in Jordan canonical form for all $u \in \CC$.
\end{dfn}

Note that while each fiber $L_u$ of a choice of exponents can be put into Jordan canonical form, the existence of an analytic choice of change of basis matrix $P$ putting $L$ simultaneously into Jordan canonical form at all points of $U$ is not guaranteed.

\begin{dfn}
  Let $\rho$ be a family of representations of $\Gamma$ on a set $U$. Define
  \[
  \Pos(\rho) \df \bigcap_{u\in U} \Pos(\rho_u).
\]
\end{dfn}

\begin{rmk}
It is frequently the case that $\Pos(\rho)$ is nonempty for nontrivial families $\rho$. For example, if $U$ is connected and $\rho$ is irreducible, then $\rho(-I) = \pm I$ is constant on $U$, and so $\Pos(\rho)$ is the set of all positive definite matrices in $\Herm_d$.
\end{rmk}

\begin{prop}
  \label{p:convergence}
  Let $U \subseteq \CC^m$ be open, let $\rho$ be a family of representations of $\Gamma$ on $U$, let $L$ be a corresponding choice of exponents, and suppose that $P$ puts $L$ into Jordan canonical form. Then for each compact subset $K\subseteq U$, there exists a real number $A$ depending on $K$, $\rho|_K$, $L|_K$ and $P|_K$, such that the following hold:
  \begin{enumerate}
  \item for each $u \in K$, the series $H(\rho_u,L_u,h,\tau,s)$ converges uniformly and absolutely to a smooth function of $\tau$ for all $s \in \CC$ with $\Re(s) > A$, and for all $h \in \Pos(\rho)$;
  \item this function is holomorphic as a function of $s$ and $u$, and real analytic as a function of $h$.
  \end{enumerate}
\end{prop}
\begin{proof}
  The proof is similar to the proof of Theorem 3.3 of \cite{KnoppMason1}, though there are enough differences in the statements of these results that we repeat some details. If $M$ is a matrix, then let $\norm{M}$ denote the supremum norm on its entries. Then we have
  \begin{align*}
  \norm{H(h,\tau,s)} \leq 1+\norm{h}\sum_{c=1}^\infty\sum_{\substack{d\in\ZZ\\ \gcd(c,d)=1}} \norm{\rho(\gamma)}^2\norm{\be(-L(\gamma \cdot \tau))}^2\frac{y^s}{\abs{c\tau+d}^{2s}}
  \end{align*}
  where $\gamma = \stwomat abcd$ for some choice of $a,b\in\ZZ$. After possibly replacing $\gamma$ by $T^n\gamma$, we may assume that $0\leq \Re(\gamma \cdot \tau) < 1$. We then have the basic estimate,
  \[\abs{\gamma \cdot \tau}^2 \leq 1+\Im(\gamma \cdot \tau)^2=1+\frac{y^2}{((cx+d)^2+c^2y^2)^{2}}\leq 1+\frac{1}{c^4y^2}.\]

  Use the existence of the Jordan canonical form on $U$ to write $-L$ in the form $-L=P(D+N)P^{-1}$ where $D$ is diagonal, $DN=ND$, and $N^d=0$, where $d=\dim \rho$. We obtain the estimate
  \begin{align*}
    \norm{\be(-L(\gamma \cdot \tau))}^2 \leq &\norm{P}^2\norm{P^{-1}}^2\norm{\be(D(\gamma \cdot \tau))}^2\norm{\be(N(\gamma \cdot \tau))}^2\\
    \leq &\norm{P}^2\norm{P^{-1}}^2\norm{\be(D(\gamma \cdot \tau))}^2\sum_{k=0}^{d-1}\frac{(2\pi)^k}{k!}\norm{N}^{2k}\abs{\gamma \cdot \tau}^{2k}
  \end{align*}
  Thus, if we set $C_1 = e\max(1,\norm{P}^2,\norm{P^{-1}}^2, \norm{N},\norm{N^2},\ldots, \norm{N^{d-1}})$ then we deduce that
\[    \norm{\be(-L(\gamma \cdot \tau))}^2 \leq C_1\norm{\be(D(\gamma \cdot \tau))}^2e^{\frac{2\pi}{c^4y^2}}.
  \]
  To continue, write $D = U+iV$ where $U$ and $V$ are real diagonal matrices, so that in particular $UV=VU$. Then
  \begin{align*}
    \norm{\be(D(\gamma \cdot \tau))}^2 \leq &\norm{\be(U(\gamma \cdot \tau))}^2\norm{\be(iV(\gamma \cdot \tau))}^2\\
    =&\norm{\be(iU\Im(\gamma \cdot \tau))}^2\norm{\be(iV\Re(\gamma \cdot \tau))}^2\\
    =&\norm{e^{-\frac{2\pi yU}{\abs{c\tau+d}^2}}}^2\norm{e^{-2\pi V\Re(\gamma \cdot \tau)}}^2\\
       \leq & C_2\norm{e^{-\frac{2\pi yU}{\abs{c\tau+d}^2}}}^2
  \end{align*}
  where $C_2 = \max(1,\norm{e^{-2\pi V}}^2)$. Putting these estimates together, we have shown that if we normalize our representatives $\gamma$ for cosets in $\langle \pm T\rangle \backslash \Gamma$ such that $0\leq \Re(\gamma\tau) < 1$, then
    \begin{equation}
      \label{eq:estimate1}
      \norm{H(h,\tau,s)} \leq 1+C_1C_2\norm{h}y^s\sum_{c=1}^\infty\sum_{\substack{d\in\ZZ\\ \gcd(c,d)=1}} \norm{\rho(\gamma)}^2\norm{e^{-\frac{2\pi yU}{\abs{c\tau+d}^2}}}^2\frac{e^{\frac{2\pi}{c^4y^2}}}{\abs{c\tau+d}^{2s}}.
    \end{equation}

    It remains to estimate $\norm{\rho(\gamma)}$. For this one can use Corollary 3.5 of \cite{KnoppMason4} to obtain an estimate for this term that is polynomial in $c^2+d^2$, with constants and degree that depend only on $\rho$. Since the exponential factors in equation \eqref{eq:estimate1} converge to $1$ as $c$ grows, this estimate is sufficient to prove part (1) of the proposition. All constants in these various estimates can be chosen uniformly on $K$, so that (2) follows as well.
  \end{proof}
\begin{rmk}
The case of a single representation can be deduced from Proposition \ref{p:convergence} by considering a constant family of representations on $\CC$. One can always put a constant family of exponents into Jordan canonical form, so the hypothesis that such a Jordan canonical form exists can be ignored when considering individual representations.
\end{rmk}

\section{Tame harmonic metrics}
\label{s:harmonic}
Nonabelian Hodge theory describes a correspondence between categories of representations of fundamental groups, and categories of \emph{Higgs bundles} on the underlying base manifold. A key ingredient for establishing such correspondences lies in proving the existence of metrics satisfying a nonlinear differential equation as in the following definition.
\begin{dfn}
  \label{d:harmonic}
  A Hermitian positive definite metric $H \colon \uhp \to M_d(\CC)$ for a representation $\rho$ is said to be \emph{harmonic} provided that
  \[
  \partial \bar \partial \log(H) = \tfrac 12 [\bar\partial \log(H),\partial \log(H)]
\]
where $\partial \log(H) = H^{-1}\partial(H)$, $\bar\partial \log(H) = H^{-1}\bar\partial(H)$.
\end{dfn}

\begin{ex}
  If $\rho$ is a one-dimensional character (necessarily unitary in the case $\Gamma = \SL_2(\ZZ)$), then the commutator in Definition \ref{d:harmonic} vanishes, and the harmonicity condition simplifies to $\log(H)$ being harmonic in the usual sense. If more generally $\rho$ is unitary, the constant map $H=I$ satisfies the necessary invariance property in this case, and it defines a harmonic metric. This is the usual Petersson inner product for unitary representations.
\end{ex}

\begin{lem}
Let $H$ be a harmonic metric for $\rho$, and let $g \in \GL_d(\CC)$. Then $g^tH\bar g$ is a harmonic metric for $g^{-1}\rho g$. 
\end{lem}
\begin{proof}
  First since $\rho(\gamma)^tH(\gamma \tau)\overline{\rho(\gamma)} = H(\tau)$, we find that
\begin{align*}
  g^tH(\tau)\bar g &= g^t\rho(\gamma)^tH(\gamma \tau)\overline{\rho(\gamma)g}\\
                   &= (g^{-1}\rho(\gamma)g)^tg^tH(\gamma\tau)\bar g (\overline{g^{-1}\rho(\gamma)g})
\end{align*}
Next notice that
\[
  \partial\log(g^tH\bar g) =(g^tH\bar g)^{-1}\partial(g^tH\bar g)=\bar g^{-1}H^{-1}g^{-t}g^t\partial(H)\bar g = \bar g^{-1}\partial \log(H) \bar g
\]
and similarly for $\bar\partial \log (g^tH\bar g)$. Therefore,
\begin{align*}
  \partial\bar\partial\log(g^tH\bar g ) &=\bar g^{-1}\partial\bar\partial \log(H) \bar g\\
                                        &= \tfrac 12\bar g^{-1}[\bar\partial \log(H),\partial \log(H)] \bar g\\
                                        &= \tfrac 12\bar [\bar g^{-1}\bar\partial \log(H)\bar g,\bar g^{-1}\partial \log(H) \bar g]\\
  &=\tfrac 12[\bar\partial \log(g^tH\bar g),\partial \log(g^tH\bar g)].
\end{align*}
\end{proof}

\begin{dfn}
  A metric $H \colon \uhp \to M_d(\CC)$ for a representation $\rho$ is said to be \emph{tame}, or of \emph{slow growth}, for a choice of exponents $L$ provided that there exist constants $C$, $N$ such that
  \[
  \norm{\be(L^t\tau)H(\tau)\overline{\be(L\tau)}} \leq Cy^N.
  \]
\end{dfn}

Tameness arises naturally in \cite{Simpson1} when considering correspondences between stable connections and stable Higgs bundles. In general it can be difficult to write down explicit examples of harmonic metrics. Our interest in Eisenstein metrics is that they provide analytic families of metrics with appropriate invariance properties under the action of $\Gamma$. The question is whether some specialization, residue, or some other metric derived from $H(\tau,h,s)$, could satisfy the harmonicity and tameness conditions. Moreover, since the formation of Eisenstein metrics is well-adapted to working with families of representations, one might obtain universal familes of harmonic metrics living over moduli spaces of Higgs bundles. At present no general results in this direction are known, though we discuss some preliminary examples below.

\section{The inclusion representation}
\label{s:inclusion}
Suppose that $\rho \colon \Gamma \hookrightarrow \GL_2(\CC)$ is the inclusion representation. In this case a harmonic metric is known to be
\begin{equation}
\label{eq:geodesic}
  K(\tau) = \frac{1}{y}\twomat{1}{-x}{-x}{x^2+y^2}
\end{equation}
where $\tau = x+iy$. This case is rather special, since $\rho$ in fact extends to the ambient Lie group $\SL_2(\RR)$. In fact, this metric is an example of a \emph{totally geodesic metric} as in Example 4.4 of \cite{FrancRayan}, or Example 14.1.2 of \cite{CarlsonEtAl}. In this section we show that this metric $K(\tau)$ arises as a residue of Eisenstein metrics.

Notice that since in this case $\rho(-I) = -I$, the set $\Pos(\rho)$ is the set of all positive definite matrices in $\Herm_2$. The possible exponent choices $L$ take the form
\[2\pi i L = \twomat {2\pi i n}{1}{0}{2\pi i n}\]
for integers $n\in\ZZ$. Fix $n \in \ZZ$ and observe that $\be(L\tau)=q^n\stwomat{1}{\tau}{0}{1}$, where $q = \be(\tau)$. Thus, if for our choice of $h\in \Pos(\rho)$ we take $h=I$ then
  \[
    h(\tau) = \abs{q}^{2n}\twomat {1}{0}{-\tau}{1}\twomat{1}{-\overline\tau}{0}{1}=e^{-4\pi ny}\twomat{1}{-\bar\tau}{-\tau}{1+\abs{\tau}^2}.
  \]
Write $H(\tau,s) = H(\rho,L,I,\tau,s)$ and compute:
  \begin{align*}
    H(\tau,s) &= y^sh(\tau)+\sum_{c\geq 1}\sum_{\gcd(c,d) =1} \twomat{a}{c}{b}{d}\twomat{1}{0}{-\gamma \tau}{1}\twomat{1}{-\gamma \bar\tau}{0}{1}\twomat abcd e^{-4\pi n \frac{y}{\abs{c\tau+d}^2}}\frac{y^s}{\abs{c\tau+d}^{2s}}\\
              &= y^sh(\tau)+\sum_{c\geq 1}\sum_{\gcd(c,d) =1} \twomat{a-c\gamma \tau}{c}{b-d\gamma\tau}{d}\twomat{a-c\gamma\bar\tau}{b-d\gamma\bar\tau}{c}{d}e^{-4\pi n \frac{y}{\abs{c\tau+d}^2}}\frac{y^s}{\abs{c\tau+d}^{2s}}\\
              &= y^sh(\tau)+\sum_{c\geq 1}\sum_{\gcd(c,d) =1} \twomat{1}{c}{-\tau}{d}\twomat{\abs{c\tau+d}^{-2}}{0}{0}{1}\twomat{1}{-\bar\tau}{c}{d}e^{-4\pi n \frac{y}{\abs{c\tau+d}^2}}\frac{y^s}{\abs{c\tau+d}^{2s}}\\
              &= y^sh(\tau)+\sum_{c\geq 1}\sum_{\gcd(c,d) =1} \twomat{\abs{c\tau+d}^{-2}}{c}{-\tfrac{\tau}{\abs{c\tau+d}^2}}{d}\twomat{1}{-\bar\tau}{c}{d}e^{-4\pi n \frac{y}{\abs{c\tau+d}^2}}\frac{y^s}{\abs{c\tau+d}^{2s}}\\
              &= y^sh(\tau)+\sum_{c\geq 1}\sum_{\gcd(c,d) =1} \twomat{c^2+\frac{1}{\abs{c\tau+d}^{2}}}{cd-\tfrac{\bar\tau}{\abs{c\tau+d}^2}}{cd-\tfrac{\tau}{\abs{c\tau+d}^2}}{d^2+\abs{\tfrac{\tau}{c\tau+d}}^2}e^{-4\pi n \frac{y}{\abs{c\tau+d}^2}}\frac{y^s}{\abs{c\tau+d}^{2s}}\\
    &= y^sh(\tau)+\sum_{k\geq 0}\frac{(-4\pi n)^k}{k!}\sum_{c\geq 1}\sum_{\gcd(c,d) =1} \twomat{c^2+\frac{1}{\abs{c\tau+d}^{2}}}{cd-\tfrac{\bar\tau}{\abs{c\tau+d}^2}}{cd-\tfrac{\tau}{\abs{c\tau+d}^2}}{d^2+\abs{\tfrac{\tau}{c\tau+d}}^2}\frac{y^{s+k}}{\abs{c\tau+d}^{2(s+k)}}
  \end{align*}
  
  Notice that with $E(\tau,s) = y^s+\sum_{c\geq 1}\sum_{\gcd(c,d)=1}\frac{y^s}{\abs{c\tau+d}^{2s}}$ equal to the usual real-analytic Eisenstein series we have
  \begin{align}
    \label{eq:inceis1}
    \nonumber   H(\tau,s) =& e^{-4\pi ny}y^s\twomat{0}{0}{0}{1}+y^{-1}\sum_{k\geq 0}\frac{(-4\pi n)^k}{k!}E(\tau,s+k+1)\twomat{1}{-\bar\tau}{-\tau}{\abs{\tau}^2}+\\
&\sum_{k\geq 0}\frac{(-4\pi n)^k}{k!}\sum_{c\geq 1}\sum_{\gcd(c,d) =1} \twomat{c^2}{cd}{cd}{d^2}\frac{y^{s+k}}{\abs{c\tau+d}^{2(s+k)}}
\end{align}
We now focus on the final terms using a standard approach:
\begin{align*}
  &\sum_{c\geq 1}\sum_{\gcd(c,d) =1} \twomat{c^2}{cd}{cd}{d^2}\frac{y^s}{\abs{c\tau+d}^{2s}}\\
  =&\sum_{c\geq 1}c^{2-2s}\sum_{\substack{d=1\\ \gcd(c,d)=1}}^c\sum_{u\in\ZZ}\twomat{1}{\tfrac dc + u}{\tfrac dc + u}{(\tfrac dc+u)^2}\frac{y^s}{\abs{\tau+\tfrac dc+u}^{2s}}\\
=&\twomat{1}{0}{-\tau}{1}\left(\sum_{c\geq 1}c^{2-2s}\sum_{\substack{d=1\\ \gcd(c,d)=1}}^c\sum_{u\in\ZZ} \twomat{1}{\bar\tau+\tfrac dc + u}{\tau+\tfrac dc + u}{\abs{\tau+\tfrac dc+u}^2}\frac{y^s}{\abs{\tau+\tfrac dc+u}^{2s}}\right)\twomat{1}{-\bar\tau}{0}{1}
\end{align*}
If $f(\tau)$ denotes the sum over $u$ in the previous line, then $f(\tau+1)=f(\tau)$ and the Poisson summation formula may be used. We must first compute the Fourier transform of the summand terms: if $\be(z) \df e^{2\pi i z}$ then the Fourier transform is
\begin{align*}
  &y^s\int_{-\infty}^\infty \twomat{1}{x-iy+\tfrac dc + t}{x+iy+\tfrac dc + t}{\abs{x+iy+\tfrac dc+t}^2}\frac{\be(-mt)}{\abs{x+iy+\tfrac dc+t}^{2s}}dt\\
  =&y^s\be(mx + m\tfrac dc)\int_{-\infty}^\infty \twomat{1}{r-iy}{r+iy}{r^2+y^2}\frac{\be(-mr)}{(r^2+y^2)^{s}}dr.
\end{align*}
Let $g(m,y,s) = y^s\int_{-\infty}^\infty \stwomat{1}{r-iy}{r+iy}{r^2+y^2}\tfrac{\be(-mr)}{(r^2+y^2)^{s}}dr$. Then, putting everything together, we have shown that:
\begin{align}
\label{eq:inceis2}
    \nonumber   H(\tau,s) =& y^se^{-4\pi ny}\twomat{0}{0}{0}{1}+y^{-1}\sum_{k\geq 0}\frac{(-4\pi n)^k}{k!}E(\tau,s+k+1)\twomat{1}{-\bar\tau}{-\tau}{\abs{\tau}^2}+\\
 & \twomat{1}{0}{-\tau}{1}\left(\sum_{k\geq 0}\frac{(-4\pi n)^k}{k!}\sum_{m\in\ZZ}\be(mx)g(m,y,s+k)\sum_{c\geq 1}c^{2-2(s+k)}\sum_{\substack{d=1\\ \gcd(c,d)=1}}^c \be(m\tfrac dc)\right)\twomat{1}{-\bar\tau}{0}{1}
\end{align}
Recall that we have the following evaluation of the Ramanujan sum:
\[
  \sum_{\substack{d=1\\ \gcd(c,d)=1}}^c \be(m\tfrac dc)=\begin{cases}
    \sum_{g\mid \gcd(c,m)}\mu(\tfrac cg)g & m\neq 0,\\
    \phi(c) & m=0,
  \end{cases}
\]
and therefore
\[
  \sum_{c\geq 1}c^{2-2(s+k)}\sum_{\substack{d=1\\ \gcd(c,d)=1}}^c \be(m\tfrac dc)=\begin{cases}
    \frac{\sigma_{3-2(s+k)}(\abs{m})}{\zeta(2(s+k)-2)} & m\neq 0,\\
    \frac{\zeta(2(s+k)-3)}{\zeta(2(s+k)-2)}&m=0.
  \end{cases}
\]
That is, we have shown the following: $H(\tau,s) = \stwomat 10{-\tau}1 \tilde H(\tau,s) \stwomat 1{-\bar \tau}01$ where
\begin{align}
\label{eq:inceis3}
    \nonumber   &\tilde H(\tau,s) = y^se^{-4\pi ny}\twomat{0}{0}{0}{1}+y^{-1}\sum_{k\geq 0}\frac{(-4\pi n)^k}{k!}E(\tau,s+k+1)\twomat{1}{0}{0}{0}+\\
  &\sum_{k\geq 0}\frac{(-4\pi n)^k}{k!}\left(\frac{\zeta(2(s+k)-3)}{\zeta(2(s+k)-2)}g(0,y,s+k)+\sum_{\substack{m\in\ZZ\\ m \neq 0}}\frac{\be(mx)\sigma_{3-2(s+k)}(\abs{m})g(m,y,s+k)}{\zeta(2(s+k)-2)}\right)
\end{align}

To conclude the computation of the Fourier coefficients of $H(\tau,s)$ it now remains to give a more concrete expression for $g(n,y)$. Equations (3.18) and (3.19) of \cite{Iwaniec} yield:
\begin{equation}
  \label{eq:iwaniec}
\int_{-\infty}^\infty \frac{\be(-mr)}{(r^2+y^2)^{s}}dr = \begin{cases}
  \pi^{\frac 12}\frac{\Gamma(s-\frac 12)}{\Gamma(s)}y^{1-2s} & m=0,\\
  2\pi^s\Gamma(s)^{-1}\abs{m}^{s-\frac 12}y^{-s+\frac 12}K_{s-\frac 12}(2\pi \abs{m}y) & m\neq 0.
\end{cases}
\end{equation}
This allows us to evaluate the diagonal terms in $g(m,y,s)$. It remains to treat the anti-diagonal terms, and we first suppose $m\neq 0$. By the product rule, and since $\Re(s) \gg 0$,
\begin{align*}
  & \int_{-\infty}^\infty \frac{r\be(-mr)}{(r^2+y^2)^{s}}dr \\
  =&-\frac{1}{2\pi in}\lim_{N \to \infty}\left. \frac{r\be(-mr)}{(r^2+y^2)^s}\right|^N_{-N}+\frac{1}{2\pi im}\int_{-\infty}^\infty\frac{((r^2+y^2)^s-2sr^2(r^2+y^2)^{s-1})\be(-mr)}{(r^2+y^2)^{2s}}dr\\
  =&\frac{1}{2\pi i m}\int_{-\infty}^\infty\frac{\be(-mr)}{(r^2+y^2)^{s}}dr-\frac{2s}{2\pi i m}\int_{-\infty}^\infty\frac{\be(-mr)}{(r^2+y^2)^{s}}dr+\frac{2sy^2}{2\pi i m}\int_{-\infty}^\infty\frac{\be(-mr)}{(r^2+y^2)^{s+1}}dr
\end{align*}
That is, for $m\neq 0$ we've shown that:
\begin{align*}
  g(m,y,s) =& y^s\int_{-\infty}^\infty \twomat{\frac{\be(-mr)}{(r^2+y^2)^{s}}}{\left(\frac{1-2s}{2\pi i m}-iy\right)\frac{\be(-mr)}{(r^2+y^2)^{s}}+\frac{2sy^2}{2\pi i m}\frac{\be(-mr)}{(r^2+y^2)^{s+1}}}{\left(\frac{1-2s}{2\pi i m}+iy\right)\frac{\be(-mr)}{(r^2+y^2)^{s}}+\frac{2sy^2}{2\pi i m}\frac{\be(-mr)}{(r^2+y^2)^{s+1}}}{\frac{\be(-mr)}{(r^2+y^2)^{s-1}}}dr\\
  =& y^s\twomat{1}{\frac{1-2s}{2\pi i m}-iy}{\frac{1-2s}{2\pi i m}+iy}{0}2\pi^s\Gamma(s)^{-1}\abs{m}^{s-\frac 12}y^{-s+\frac 12}K_{s-\frac 12}(2\pi \abs{m}y)+\\
          &y^s\twomat{0}{\tfrac{2sy^2}{2\pi i m}}{\tfrac{2sy^2}{2\pi i m}}{0}2\pi^{s+1}\Gamma(s+1)^{-1}\abs{m}^{s+\frac 12}y^{-s-\frac 12}K_{s+\frac 12}(2\pi \abs{m}y)+\\
          &y^s\twomat{0}{0}{0}{1}2\pi^{s-1}\Gamma(s-1)^{-1}\abs{m}^{s-\frac 32}y^{-s+\frac 32}K_{s-\frac 32}(2\pi \abs{m}y).
\end{align*}
This can be cleaned up somewhat using the functional equation $\Gamma(s+1)=s\Gamma(s)$ for the gamma function:
\begin{align*}
  g(m,y,s) =&\twomat{1}{\frac{1-2s}{2\pi i m}-iy}{\frac{1-2s}{2\pi im}+iy}{0}2\pi^s\Gamma(s)^{-1}\abs{m}^{s-\frac 12}y^{\frac 12}K_{s-\frac 12}(2\pi \abs{m}y)+\\
          &\twomat{0}{\tfrac{2y^2}{2\pi im}}{\tfrac{2y^2}{2\pi i m}}{0}2\pi^{s+1}\Gamma(s)^{-1}\abs{m}^{s+\frac 12}y^{-\frac 12}K_{s+\frac 12}(2\pi \abs{m}y)+\\
          &\twomat{0}{0}{0}{1}2\pi^{s-1}(s-1)\Gamma(s)^{-1}\abs{m}^{s-\frac 32}y^{\frac 32}K_{s-\frac 32}(2\pi \abs{m}y). 
\end{align*}
In the interest of pairing terms corresponding to $\pm m$, notice that for $m > 0$:
\begin{align*}
  &\be(mx)g(m,y,s)+\be(-mx)g(-m,y,s) \\
=&\cos(2\pi mx)(g(m,y,s)+g(-m,y,s))+i\sin(2\pi mx)(g(m,y,s)-g(-m,y,s))\\
  =& \frac{4y^{\frac 12}\pi^sm^{s-\frac 12}}{\Gamma(s)}\left(\cos(2\pi mx)\twomat{K_{s-\frac 12}(2\pi my)}{-iyK_{s-\frac 12}(2\pi my)}{iyK_{s-\frac 12}(2\pi my)}{\frac{(s-1)y}{\pi m}K_{s-\frac 32}(2\pi my)}+\right.\\
  &\left.\sin(2\pi mx)\left(\frac{1-2s}{2\pi m}K_{s-\frac 12}(2\pi my)+yK_{s+\frac 12}(2\pi my)\right)\twomat 0110\right)
\end{align*}

Finally, to evaluate $g(0,y)$ it remains to observe that since $\tfrac{r}{(r^2+s^2)^s}$ is an odd function of $r$, $\int_{-\infty}^\infty \frac{r}{(r^2+y^2)^s}dr=0$. Therefore we find that
\begin{align*}
  g(0,y,s) &= y^s\int_{-\infty}^\infty \stwomat{1}{r-iy}{r+iy}{r^2+y^2}\tfrac{1}{(r^2+y^2)^{s}}dr\\
         &=y^s\int_{-\infty}^\infty \stwomat{1}{-iy}{iy}{r^2+y^2}\tfrac{1}{(r^2+y^2)^{s}}dr\\
  &=y^s\twomat{\pi^{\frac 12}\frac{\Gamma(s-\frac 12)}{\Gamma(s)}y^{1-2s}}{-iy\pi^{\frac 12}\frac{\Gamma(s-\frac 12)}{\Gamma(s)}y^{1-2s}}{iy\pi^{\frac 12}\frac{\Gamma(s-\frac 12)}{\Gamma(s)}y^{1-2s}}{\pi^{\frac 12}\frac{\Gamma(s-\frac 32)}{\Gamma(s-1)}y^{3-2s}}
\end{align*}
so that
\[
  g(0,y,s) = \frac{\pi^{\frac 12}y^{1-s}\Gamma(s-\frac 12)}{\Gamma(s)}\twomat{1}{-iy}{iy}{\frac{2s-2}{2s-3}y^{2}}
\]

Finally recall that the Fourier expansion for $E(\tau,s+1)$ is
\begin{align*}
  E(\tau,s+1) =& y^{s+1}+\frac{\pi^{2s+1}\Gamma(-s)\zeta(-2s)}{\Gamma(s+1)\zeta(2s+2)}y^{-s}+\\
  &\frac{4\pi^{s+1}y^{\frac 12}}{\Gamma(s+1)\zeta(2s+2)}\sum_{m=1}^\infty m^{-s-\frac 12}\sigma_{2s+1}(m)\cos(2\pi m x)K_{s+\frac 12}(2\pi m y).
\end{align*}
Putting these computations together allows us to prove the following:
\begin{thm}
  \label{t:incmain}
  Let $\rho$ denote the inclusion representation of $\Gamma$, let $L = \stwomat {n}{(2\pi i )^{-1}}0{n}$ for $n \in \ZZ$, and set $H(\tau,s) \df H(\rho,L,I,\tau,s)$. Then the Eisenstein metric $H(\tau,s)$ has a Fourier expansion of the form
  \[H(\tau,s) = \twomat 10{-\tau}1\left(\sum_{m\geq 0} H_m(\tau,s)\right)\twomat 1{-\bar\tau}01\]
  where 
\begin{align*}
  H_0(\tau,s) =& y^se^{-4\pi ny}\twomat{1}{0}{0}{1}+\\
  &\sum_{k\geq 0}\frac{(-4\pi n)^k}{k!}\frac{\pi^{2s+2k+1}\Gamma(-s-k)\zeta(-2s-2k)}{\Gamma(s+k+1)\zeta(2s+2k+2)}y^{-s-k-1}\twomat{1}{0}{0}{0}+\\
 &\sum_{k\geq 0}\frac{(-4\pi n)^k}{k!}\frac{\pi^{\frac 12}y^{1-s-k}\Gamma(s+k-\frac 12)\zeta(2s+2k-3)}{\Gamma(s+k)\zeta(2s+2k-2)}\twomat{1}{-iy}{iy}{\frac{2s+2k-2}{2s+2k-3}y^2}
\end{align*}
and for $m\geq 1$,
\begin{align*}
 & H_m(\tau,s) =4\pi^s\sum_{k\geq 0}\frac{(-4\pi^2 n)^k}{k!}\left(\frac{\pi\sigma_{2s+2k+1}(m)\cos(2\pi m x)K_{s+k+\frac 12}(2\pi m y)}{m^{s+k+\frac 12}\Gamma(s+k+1)\zeta(2s+2k+2)y^{\frac 12}}\twomat{1}{0}{0}{0} + \right.\\
  &\frac{y^{\frac 12}m^{s+k-\frac 12}\sigma_{3-2s-2k}(m)}{\Gamma(s+k)\zeta(2s+2k-2)}\cos(2\pi mx)\twomat{K_{s+k-\frac 12}(2\pi my)}{-iyK_{s+k-\frac 12}(2\pi my)}{iyK_{s+k-\frac 12}(2\pi my)}{\frac{(s+k-1)y}{\pi m}K_{s+k-\frac 32}(2\pi my)}+\\
  &\left.\frac{y^{\frac 12}m^{s+k-\frac 12}\sigma_{3-2s-2k}(m)}{\Gamma(s+k)\zeta(2s+2k-2)}\sin(2\pi mx)\left(\frac{1-2s-2k}{2\pi m}K_{s+k-\frac 12}(2\pi my)+yK_{s+k+\frac 12}(2\pi my)\right)\twomat 0110\right)
\end{align*}
Moreover, $H(\tau,s)$ admits meromorphic continuation to the region $\Re(s) > \tfrac 32$, and its only pole in this region is simple and located at $s=2$. This pole comes from the constant term $H_0(\tau,s)$, and the residue at $s=2$ is a tame harmonic metric for the inclusion representation:
\[
\Res_{s=2}H(\tau,s) = \frac{3}{2\pi y}\twomat{1}{-x}{-x}{x^2+y^2} = \frac{3}{2\pi}K(\tau).
\]
In particular, the residue does not depend on the choice of exponent matrix $L$.
\end{thm}
\begin{proof}
  The computation of the Fourier coefficients is accomplished by substituting our expressions for the Fourier transforms $g(m,y,s)$ into equation \eqref{eq:inceis3} and simplifying. For the meromorphic continutation, recall that $\Gamma(s/2)\zeta(s)$ has two simple poles, at $s=0$ and $s=1$, but it is otherwise holomorphic. Further, it is nonvanishing in a neighbourhood of $s=1$ and for $\Re(s) \geq 1$.

  The first term in the expression for $H_0(\tau,s)$ is holomorphic in $s$. For each summand in the second line of the expression for $H_0(\tau,s)$ consider the ratio
\[
  \frac{\Gamma(-s-k)\zeta(-2s-2k)}{\Gamma(s+k+1)\zeta(2s+2k+2)}
\]
which is $O(1)$ as a function of $k$. If $\Re(s) > \frac 32$, then $\Re(-2s-2k) < -3-2k \leq -3$. Therefore the numerator of the lined formula above is holomorphic in this region. Likewise, $\Re(2s+2k+2) > 5+2k$, so that the denominator is holomorphic and nonvanishing in this region. Therefore the second line in the description of $H_0(\tau,s)$ converges to a holomorphic function when $\Re(s) > \frac 32$.

For the final set of terms in $H_0(\tau,s)$ we consider the expressions
\[
\frac{\Gamma(s+k-\frac 12)\zeta(2s+2k-3)}{\Gamma(s+k)\zeta(2s+2k-2)}
\]
Here if $\Re(s) > \tfrac 32$ then $\Re(2s+2k-3) > 2k > 0$, so that the only possible pole can come from solutions to $2s+2k-3=1$, which is $s=2-k$. In the region $\Re(s)>\tfrac 32$ this can only occur if $k=0$, in which case a pole occurs at $s=2$. For the denominators we have that $\Re(2s+2k-2) > 1+2k > 1$, so that the denominator is holomorphic and nonvanishing. It follows that when $\Re(s) > \tfrac 32$, the constant term $H_0(\tau,s)$ is holomorphic save for a simple pole arising from the $k=0$ term in the last sum of its expression in Theorem \ref{t:incmain}.

The higher Fourier coefficients $H_m(\tau,s)$ are holomorphic in the region $\Re(s) > \tfrac 32$, and the proof is similar to the constant term. A new feature is that one must estimate sums such as
\[\sum_{k\geq 0}\frac{(4\pi^2n)^k\sigma_{2s+2k+1}(m)K_{s+k+\frac 12}(2\pi my)}{m^kk!\Gamma(s+k+1)\zeta(2s+2k+2)}.\]
When $\abs{\nu}$ is large relative to $x$, one can estimate $K_\nu(x)$ via the first few terms of its Taylor expansion (see Appendix B above (B.35) of \cite{Iwaniec}). More precisely, in the tail of the sum where $\abs{2\pi m y} \ll 1+\abs{s+k+\frac 12}^{1/2}$, one can approximate
\begin{equation}
  \label{eq:bessel}
  K_{s+k+\frac 12}(2\pi m y) \approx\tfrac{1}{2}\Gamma(s+k+\tfrac 12)(\pi m y)^{-s-k-\frac 12}.
\end{equation}
In this way one can show that for large $k$, the Bessel terms are mollified by the $\Gamma$-terms in the denominators. Standard and more elementary estimates for the remaining factors appearing in $H_m(\tau,s)$ then allow one to deduce sufficiently fast convergence to show that these sums indeed yield a holomorphic function in the region $\Re(s) > \tfrac 32$.

It remains to compute the residue at $s=2$, and by the preceding analysis we have:
\begin{align*}
  \Res_{s=2} H(\tau,s) =& \twomat 10{-\tau}1\Res_{s=2}H_0(\tau,s)\twomat{1}{-\bar\tau}01\\
  =&\twomat 10{-\tau}1\Res_{s=2}\frac{\pi^{\frac 12}y^{1-s}\Gamma(s-\frac 12)\zeta(2s-3)}{\Gamma(s)\zeta(2s-2)}\twomat{1}{-iy}{iy}{\frac{2s-2}{2s-3}y^2}\twomat{1}{-\bar\tau}01\\
  =&\frac{\pi^{\frac 12}\Gamma(\frac 32)\Res_{s=2}\zeta(2s-3)}{y\Gamma(2)\zeta(2)}\twomat 10{-\tau}1\twomat{1}{-iy}{iy}{2y^2}\twomat{1}{-\bar\tau}01\\
  =&\frac{3}{2\pi y}\twomat{1}{-x}{-x}{x^2+y^2}
\end{align*}
This concludes the proof of Theorem \ref{t:incmain}.
\end{proof}

\section{Unitary monodromy at the cusp}
\label{s:unitary}
Now let $\rho$ be a representation with $\rho(T)$ unitary. Then we can write $\rho(T) = e^{2\pi i L}$ where $L$ is Hermitian and real. In fact, we may and shall suppose that $\rho(T)$ and $L$ are both diagonal.
\begin{lem}
  Suppose that $\rho(T)$ and $L$ are both diagonal. If $h\in\Pos(\rho)$ then $h$ satisfies
  \[
  \be(-Lz)h\be(Lz)= h
\]
for all $z\in\CC$.
\end{lem}
\begin{proof}
The lemma holds for $z=1$ by the hypothesis $h \in \Pos(\rho)$. Without loss of generality we can suppose that the distinct eigenvalues of $L$ are $r_1,\ldots, r_m$ each with corresponding multiplicity $\mu_j$ for $j=1,\ldots, m$. The commutator algebra of $\rho(T) = \be(L)$ consists of the Levi subalgebra of block diagonal matrices with blocks of size $(\mu_1,\ldots ,\mu_m)$. The commutant of $\be(Lz)$ for $z\in\CC$ can only possibly increase in size (e.g. if $z=0$). This proves the lemma.
\end{proof}

Write $\tau=x+iy$, so that for $h\in\Pos(\rho)$, the preceding lemma implies that
\[
  h(\tau) = e^{-2\pi i L(x+iy)}h e^{2\pi i L(x-iy)} = e^{4\pi Ly}h.
\]
By abuse of notation below, $a,b$ denote integers chosen so that the corresponding matrix is unimodular; in particular, $b$ changes in the last line of the following computation, but the result is independent of this choice, which justifies our abuse of notation. With this point made, we compute:
\begin{align*}
  &  H(h,\tau,s)\\
  &= e^{4\pi Ly}y^sh + \sum_{c\geq 1}\sum_{\gcd(c,d)=1} \rho\stwomat abcd^t e^{4\pi L \frac{y}{\abs{c\tau+d}^2}}h\overline{\rho\stwomat abcd}\frac{y^s}{\abs{c\tau+d}^{2s}}\\
  &= e^{4\pi Ly}y^sh + \sum_{c\geq 1}\sum_{\substack{d=1\\\gcd(c,d)=1}}^c\sum_{r\in\ZZ} \rho\stwomat abc{cr+d}^te^{4\pi L \frac{y}{\abs{c(\tau+r)+d}^2}}h\overline{\rho\stwomat abc{cr+d}}\frac{y^s}{\abs{c(\tau+r)+d}^{2s}}\\
  &=e^{4\pi Ly}y^sh + \sum_{c\geq 1}\sum_{\substack{d=1\\\gcd(c,d)=1}}^c\sum_{r\in\ZZ}\sum_{k\geq 0} \frac{(4\pi)^k}{k!}\rho\stwomat abc{cr+d}^t L^kh\overline{\rho\stwomat abc{cr+d}}\frac{y^{s+k}}{\abs{c(\tau+r)+d}^{2(s+k)}}\\
  &=e^{4\pi Ly}y^sh + y^s\sum_{k\geq 0} \frac{(4\pi y)^k}{k!}\sum_{c\geq 1}\frac{1}{c^{2(s+k)}}\sum_{\substack{d=1\\\gcd(c,d)=1}}^c\sum_{r\in\ZZ}\frac{\rho(T^r)\rho\stwomat abc{d}^t L^kh\overline{\rho\stwomat abc{d}}\rho(T^{-r})}{\abs{\tau+r+\frac{d}{c}}^{2(s+k)}}
\end{align*}
where in the last line we have used that $\rho(T)$ is diagonal to drop the transpose. Likewise, since $L$ is real, this means $\rho(T)$ is unitary and we have used the identity $\overline{\rho(T^r)} = \rho(T^{-r})$.

Let $G(\tau)$ denote the sum over $r$ above. Notice that $G(\tau+1) = \rho(T)^{-1}G(\tau)\rho(T)$ so that if we write
\[
  \tilde G(\tau) = \be(Lx)G(\tau)\be(-Lx)
\]
then $\tilde G(\tau+1)=\tilde G(\tau)$. For simplicity, to study $\tilde G$ we momentarily write
\begin{align*}
  M =& \rho\stwomat abc{d}^t L^kh\overline{\rho\stwomat abc{d}},\\
  L =& \diag(e_1,\ldots, e_n),
\end{align*}
for real exponents $e_j$. Then the $(i,j)$-entry of $\tilde G$ is:
\begin{align*}
  \tilde G_{ij} =& M_{ij}\sum_{r\in\ZZ}\be((e_i-e_j)(x+r))\frac{1}{((x+r+\frac dc)^2+y^2)^{s+k}}\\
=& \be\left((e_j-e_i)\frac dc\right) M_{ij}\sum_{r\in\ZZ}\be((e_i-e_j)(X+r))\frac{1}{((X+r)^2+y^2)^{s+k}}
\end{align*}
where $X = x+\frac dc$. We can evaluate this last sum using Poisson summation: with the Fourier transform
\begin{align*}
  f(u) =& \int_{-\infty}^\infty \be((e_i-e_j)(X+r))\frac{1}{((X+r)^2+y^2)^{s+k}}\be(-ur)dr\\
  =&\be(Xu)\int_{-\infty}^\infty \be((e_i-e_j-u)r)\frac{1}{(r^2+y^2)^{s+k}}dr
\end{align*}
Poisson summation gives
\[
  \tilde G_{ij} =\be\left((e_j-e_i)\frac dc\right) M_{ij}\sum_{u\in\ZZ} f(u).
\]

Formulas (3.18) and (3.19) of \cite{Iwaniec} yield expressions
\begin{equation}
  \label{eq:ft}
f(u) = \begin{cases}
  \pi^{\frac 12}\be((x+\frac dc)u)\frac{\Gamma(s+k-\frac 12)}{\Gamma(s+k)}y^{1-2(s+k)} & e_i-e_j=u,\\
  2\frac{\pi^{s+k}\be((x+\frac dc)u)\abs{u+e_j-e_i}^{s+k-\frac 12}}{y^{s+k-\frac 12}\Gamma(s+k)}K_{s+k-\frac 12}(2\pi \abs{u+e_j-e_i}y)&e_i-e_j \neq u.
\end{cases}
\end{equation}
Thus, if
\[
  z_{ij} = \begin{cases}
    1 & e_i-e_j \in \ZZ,\\
    0 & e_i-e_j \not \in \ZZ,
  \end{cases}
  \]
then we deduce that
  \begin{align*}
  \tilde G_{ij} =& \pi^{\frac 12}\be((e_i-e_j)x)\frac{\Gamma(s+k-\frac 12)}{\Gamma(s+k)}y^{1-2(s+k)}M_{ij}z_{ij}+\\
  &\frac{2\pi^{s+k}\be\left((e_j-e_i)\frac dc\right) M_{ij}}{y^{s+k-\frac 12}\Gamma(s+k)}\sum_{\substack{u\in\ZZ\\ u\neq e_i-e_j}} \be((x+\tfrac dc)u)\abs{u+e_j-e_i}^{s+k-\frac 12}K_{s+k-\frac 12}(2\pi \abs{u+e_j-e_i}y)
\end{align*}
Notice that
\[
  G_{ij} = (\be(-Lx)\tilde G \be(Lx))_{ij} =\be((e_j-e_i)x)\tilde G_{ij}
\]
Therefore, putting all of this together, we have shown that
\begin{align*}
  H_{ij} =& e^{4\pi e_iy}y^sh_{ij}+\pi^{\frac 12}y^{1-s}\sum_{k\geq 0} \frac{(4\pi)^k}{k!}\sum_{c\geq 1}\frac{1}{c^{2(s+k)}}\sum_{\substack{d=1\\\gcd(c,d)=1}}^c\frac{\Gamma(s+k-\frac 12)}{y^k\Gamma(s+k)}M_{ij}z_{ij}+\\
          &y^{\frac 12}\sum_{k\geq 0} \frac{(4\pi)^k}{k!}\sum_{c\geq 1}\frac{1}{c^{2(s+k)}}\sum_{\substack{d=1\\\gcd(c,d)=1}}^c\frac{2\pi^{s+k} M_{ij}}{\Gamma(s+k)}\times \\
  &\sum_{\substack{u\in\ZZ\\ u\neq e_i-e_j}} \be((x+\tfrac dc)(u+e_j-e_i))\abs{u+e_j-e_i}^{s+k-\frac 12}K_{s+k-\frac 12}(2\pi \abs{u+e_j-e_i}y)
\end{align*}

Rearranging terms yields the following Fourier expansion for the $(i,j)$-entry of $H(h,\tau,s)$:
\begin{align}
  \label{eq:fourier1}
\nonumber  H_{ij} =&  e^{4\pi e_iy}y^sh_{ij}+\pi^{\frac 12}y^{1-s}\frac{\Gamma(s-\frac 12)}{\Gamma(s)}\sum_{c\geq 1}\frac{1}{c^{2s}}\left(\sum_{\substack{d=1\\\gcd(c,d)=1}}^c\rho\stwomat abcd^t{}_1F_1(s-\tfrac 12,s,\tfrac{4\pi L}{c^2y})h\overline{\rho\stwomat abcd}\right)_{ij}z_{ij}+\\
                   &2y^{\frac 12}\pi^s\sum_{\substack{u\in\ZZ\\ u\neq e_i-e_j}}\abs{u+e_j-e_i}^{s-\frac 12}\be((u+e_j-e_i)x)\sum_{c\geq 1}\frac{1}{c^{2s}}\sum_{\substack{d=1\\\gcd(c,d)=1}}^c\be((u+e_j-e_i)\tfrac dc)\\
  \nonumber & \left(\rho\stwomat abcd^t \sum_{k\geq 0}\left(\frac{K_{s+k-\frac 12}(2\pi \abs{u+e_j-e_i}y)}{k!\Gamma(s+k)}\left(\frac{4\pi^2\abs{u+e_j-e_i}L}{c^2}\right)^k\right)h\overline{\rho\stwomat abcd}\right)_{ij}
\end{align}
Note that the convergence of the sum on $k$ is deduced as in the equation \eqref{eq:bessel} in the proof of Theorem \ref{t:incmain}, which uses the expansion of $K_{s+k-\frac 12}$ near zero when $k$ is large.

The first line in equation \eqref{eq:fourier1} gives the constant term of $H(\tau,s)$, which in particular is independent of $x$, unlike for the inclusion representation in Section \ref{s:inclusion}. Inspired by the discussion in Section \ref{s:inclusion}, it is natural to consider the righmost pole of this constant term (if such a pole exists!) and its corresponding residue. Restrict to the case where $e_i-e_j\not \in \ZZ$ unless $i=j$, which is a familiar condition from the study of ordinary differential equations. This condition implies that $z_{ij} = \delta_{ij}$. Since the term $e^{4\pi Ly}y^sh$ is entire as a function of $s$, we are interested in the diagonal terms of the matrix valued function:
\[
C(\tau,s)=\pi^{\frac 12}y^{1-s}\frac{\Gamma(s-\frac 12)}{\Gamma(s)}\sum_{c\geq 1}\frac{1}{c^{2s}}\left(\sum_{\substack{d=1\\\gcd(c,d)=1}}^c\rho\stwomat abcd^t{}_1F_1(s-\tfrac 12,s,\tfrac{4\pi L}{c^2y})h\overline{\rho\stwomat abcd}\right)
\]
Unfortunately the Kloosterman sums appearing above are somewhat unwieldy to handle via a direct approach in general. However, basic estimates show that the rightmost pole arises from the constant term of ${}_1F_1(s-\tfrac 12,s,\tfrac{4\pi L}{c^2y})$ in its Taylor expansion in $4\pi L/c^2y$, so that one is really interested in the analytic properties of the diagonal terms of
\[
  C_0(\tau,s) = \pi^{\frac 12}y^{1-s}\frac{\Gamma(s-\frac 12)}{\Gamma(s)}\sum_{c\geq 1}\frac{1}{c^{2s}}\left(\sum_{\substack{d=1\\\gcd(c,d)=1}}^c\rho\stwomat abcd^th\overline{\rho\stwomat abcd}\right)
\]

This expression is a little more manageable. For example, suppose that $h=I_d$ and $\rho$ is unitary, so that $\rho(\gamma)^t\overline{\rho(\gamma)}=I_d$. In this case
\[
  C_0(\tau,s) = \pi^{\frac 12}y^{1-s}\frac{\Gamma(s-\frac 12)}{\Gamma(s)}\sum_{c\geq 1}\frac{\phi(c)}{c^{2s}}I_d=\pi^{\frac 12}y^{1-s}\frac{\Gamma(s-\frac 12)\zeta(2s-1)}{\Gamma(s)\zeta(2s)}I_d
\]
It follows that the rightmost pole occurs at $s=1$, and the residue is a multiple of the identity matrix. Therefore, this shows that when $\rho$ is unitary, the rightmost pole of $H(\tau,s)$ for $h=I_d$ occurs at $s=1$ and the residue is a multiple of the Petersson inner product, which is a harmonic metric for trivial reasons. It is unclear how general this phenomenon is. In the next section we discuss an example where $\rho(T)$ is unitary but $\rho$ is not unitarizable, to indicate some of the difficulties of analyzing expressions like $C_0(\tau,s)$ in more general circumstances.

\begin{rmk}
  Since $L$ is diagonal, the two matrix sums
  \begin{align*}
&\sum_{\substack{d=1\\\gcd(c,d)=1}}^c\rho\stwomat abcd^th\overline{\rho\stwomat abcd},&& \sum_{\substack{d=1\\\gcd(c,d)=1}}^c\be(-L\tfrac dc)\rho\stwomat abcd^th\overline{\rho\stwomat abcd}\be(L\tfrac dc)
  \end{align*}
  have the same diagonal entries. The matrix on the right, however, only depends on $d \mod{c}$ (this observation uses that $L$ is real), whereas the matrix expression on the left undergoes a monodromy transformation after changing the values of $d$ mod $c$. Thus, the right sum involving $\be(\pm L\tfrac dc)$ could be used in the definition of $C_0(\tau,s)$, giving a more natural expression. This would also give an expression for the constent term of $H(\tau,s)$ that is more uniform with the higher Fourier coefficients, which already incoroporate such exponential factors. In the following section we consistently work with Kloosterman sums that include these additional exponential factors.
\end{rmk}

\section{An indecomposable family}
\label{s:indecomposable}
Let $\chi$ be the character of the modular form $\eta^4$, where $\eta$ is the Dedekind eta function. For $z \in \uhp$ and $\alpha \in \CC$, define a $\CC$-valued group cocycle on $\Gamma$ by the integral
\[
  \kappa(\gamma) = \int_{z}^{\gamma z}\alpha \eta^4(\tau)d\tau.
\]
This satisfies the cocycle identity
\[
  \kappa(\gamma_1\gamma_2) = \chi(\gamma_1)\kappa(\gamma_2)+\kappa(\gamma_1)
\]
for all $\gamma_1,\gamma_2 \in \Gamma$. Changing $z$ adjusts $\kappa$ by a coboundary. We can use this cocycle to define a representation of $\Gamma$ that contains a nontrivial subrepresentation, but which is not completely reducible into a direct sum of irreducible representations:
\[
  \rho(\gamma) = \twomat{\chi(\gamma)}{\kappa(\gamma)}{0}{1}.
\]
This defines a family of representations in the parameters $z$ and $\alpha$ defining $\kappa$. If $\zeta = e^{2\pi i/6}$, then
\begin{align*}
\rho(T) &= \twomat{\zeta}{\kappa(T)}01, & \rho(S) &= \twomat{-1}{\kappa(S)}01, & \rho(-1) &= \twomat 1001.
\end{align*}
Observe that $\rho(T)^6=I$, so that $\rho(T)$ is diagonalizable. However, it is not possible to diagonalize $\rho(T)$ while keeping $\rho(S)$ diagonal. From this one sees that $\rho$ contains a nontrivial subrepresentation, but it is not completely reducible for generic values of $\alpha$ and $z$.

Consider the representation of $\Gamma$ on the real vector space $\Herm_2$ defined by
\[
  M\cdot \gamma = \rho(\gamma)^tM\overline{\rho(\gamma)}.
\]
Let $U=\Herm_2^{\Gamma}$ denote the invariants for this action. A simple computation shows that generically $U$ is spanned by $\stwomat 0001$. In particular, $U$ does not contain any positive definite matrices, so that there does not exist a harmonic metric for $\rho$ that is constant as a function of $\tau$ for generic choices of $\kappa$.

To apply the material of Section \ref{s:unitary} in our study of metrics for this representation, it will be necessary to change basis so that the $T$-matrix is diagonal. If we set $\zeta = e^{2\pi i/6}$, $P=\stwomat 0{\zeta}{1}{-\zeta\kappa(T)}$, and $\psi = P\rho P^{-1}$, then one checks that
\begin{align*}
\psi(T) &= \twomat 100 \zeta, &\psi(S) &= \twomat 10{(1-\zeta)\kappa(S)-2\kappa(T)}{-1},
\end{align*}
and more generally
\[
  \psi(\gamma) = \twomat{1}{0}{(1-\zeta)\kappa(\gamma)+(\chi(\gamma)-1)\kappa(T)}{\chi(\gamma)}.
\]
The identity
\begin{equation}
\label{eq:kappaderivative}
  \frac{d\kappa}{dz}(\gamma) = (\chi(\gamma)-1)\alpha\eta^4(z)
\end{equation}
implies that $\frac{d\psi}{dz}=0$, so that the change of basis has made $\psi$ independent of $z$. It is thus a one-parameter family of representations determined by the choice of $\alpha \in \CC$ in the definition of $\kappa$. The specializations of this family are not completely reducible unless $\alpha=0$.

Turning to the associated Eisenstein metrics, we take for our exponent matrix $L = \stwomat 000{\frac 16}$. Observe that since $\psi(T)$ is diagonal with distinct eigenvalues, $\Herm_2^{\psi(T)}$ consists of real diagonal matrices. Therefore $\Pos(\psi)$ consists of positive real diagonal matrices, and since the choice of $h \in \Pos(\psi)$ only really depends on $h$ up to scaling, we can write
\[h = \twomat{1}00{A}\]
for $A \in \RR_{>0}$. With these choices of parameters we write $H(\tau,s) = H(\psi,L,h,\tau,s)$, whose Fourier expansion is given by equation \eqref{eq:fourier1} on page \pageref{eq:fourier1}.

As usual, much of the difficulty in studying the Fourier expansion of $H(\tau,s)$ lies in understanding the Kloosterman sums\footnote{For simplicity we focus on the constant term $u=0$ of the Fourier expansion, otherwise we should incorporate an additional exponential factor in the Kloosterman sum} and their associated generating series:
\begin{align*}
  \Kl(c) \df& \sum_{\substack{d=1\\\gcd(c,d)=1}}^c\be(-L\tfrac{d}{c})\psi\twomat abcd^t\twomat 100A\overline{\psi\twomat abcd}\be(L\tfrac{d}{c}),\\
  D(s) \df& \sum_{c\geq 1}\frac{\Kl(c)}{c^s}.
\end{align*}
These families of Kloosterman sums admit a second-order Taylor expansion centered on the reducible specializations of $\psi$ satisfying $\kappa(S)=2\zeta\kappa(T)$ (which we have seen is equivalent to the condition $\alpha=\kappa=0$ in the definition of $\kappa$):
\begin{prop}
  \label{p:taylor}
  There exist sequences $a_c\in\ZZ_{\geq 0}$ and $b_c\in \ZZ[e^{2\pi i/6c}]$, independent of $\kappa$, such that
    \[
\Kl(c) = \twomat{\kappa(S)-2\zeta\kappa(T)}{0}{0}{1}\twomat{a_c}{b_c}{\overline{b_c}}{0}\twomat{\overline{\kappa(S)-2\zeta\kappa(T)}}{0}{0}{1}A+\phi(c)\twomat 100A.
\]
for all $c\geq 1$.
\end{prop}
\begin{proof}
  Write $\gamma = \stwomat abcd \in \Gamma$, and observe that the lower-triangular form of $\psi(T)$ and $\psi(S)$ show that after writing $\gamma$ as a word in $S$ and $T$, we have
  \[
  \psi(\gamma) = \twomat 10{\lambda a(\gamma)}{\chi(\gamma)}
  \]
  where $\lambda =  \kappa(S)-2\zeta\kappa(T)$ and $a(\gamma) \in \ZZ[\zeta]$. Further, $a(\gamma)$ is independent of $\kappa$, as it only depends on how one writes $\gamma$ as a word in $S$ and $T$. Equivalently, this independence can be seen by writing $a(\gamma)$ as a ratio of linear combinations of values of $\kappa$. Then $a(\gamma)$ is seen to be independent of $z$ by differentiation, via equation \eqref{eq:kappaderivative}. The possible dependence of $a(\gamma)$ on $\alpha$ cancels in the ratio defining $a(\gamma)$, so that it is indeed entirely independent of the choice of $\kappa$.
  
  If we write $\omega = e^{2\pi i /c}$, then the general term in the sum defining $\Kl(c)$ takes the form
  \begin{align*}
    &\be(-L\tfrac{d}{c})\psi\twomat abcd^t\twomat 100A\overline{\psi\twomat abcd}\be(L\tfrac{d}{c})\\
    =&\twomat{1}{0}{0}{w^{-d}}\twomat 1{\lambda a(\gamma)}0{\chi(\gamma)}\twomat100A\twomat 10{\overline{\lambda a(\gamma)}}{\overline{\chi(\gamma)}}\twomat{1}{0}{0}{\omega^{d}}\\
    =&\twomat{1}{A\lambda a(\gamma)}{0}{A\omega^{-d}\chi(\gamma)}\twomat{1}{0}{\overline{\lambda a(\gamma)}}{\overline{\chi(\gamma)}\omega^d}\\
    =&\twomat{1+A\abs{\lambda a(\gamma)}^2}{A\lambda a(\gamma)\overline{\chi(\gamma)}\omega^d}{A\overline{\lambda a(\gamma)}\chi(\gamma)\omega^{-d}}{A}
  \end{align*}
  Therefore, from this expression we see that the Proposition holds with
  \begin{align*}
    a_c &= \sum_{\substack{d=1\\\gcd(c,d)=1}}^c \abs{a(\gamma)}^2, & b_c&= \sum_{\substack{d=1\\\gcd(c,d)=1}}^ca(\gamma)\overline{\chi(\gamma)}\omega^d.
  \end{align*}
\end{proof}

Values of $a_c$ and $b_c$ from Proposition \ref{p:taylor} are listed in Table \ref{t:values} on page \pageref{t:values}, and plots of the values $a_c/\phi(c)$ and $\abs{b_c}/\phi(c)$ are contained in Figure \ref{f:plots} on page \pageref{f:plots}. Polynomial growth estimates can be obtained for both $a_c$ and $b_c$ using Eichler length estimates as in Corollarly 3.5 of \cite{KnoppMason4}, though establishing an exact abscissa of convergence for $D(s)$, and the computation of the corresponding residue, would require a more detailed analysis. Completion of such analysis would likely enable one to establish analytic continuation of  $H(\tau,s)$ around its rightmost pole and compute the corresponding residue. The analysis of Section \ref{s:unitary} shows that it is really the diagonal terms of $D(s)$ that intervene in this residue computation, so that it is the sequence $a_c$ that is most important for this analysis. We leave this computation, and whether the resulting computation produces a harmonic metric for these non-unitary representations, as an open question for future investigation.

\begin{center}
  \begin{table}
    \small
    \renewcommand{\arraystretch}{1.2}
    \begin{tabular}{|lll|lll|lll|}
      \hline
      $c$ & $a_c$ & $\abs{b_c}$ & $c$ & $a_c$ & $\abs{b_c}$ & $c$ & $a_c$ & $\abs{b_c}$ \\
      \hline
$1$ & $1$ & $1.00000000000000$ & $41$ & $268$ & $64.6571107195962$ & $81$ & $408$ & $85.1656851140978$ \\
$2$ & $3$ & $1.73205080756888$ & $42$ & $72$ & $20.0687840963522$ & $82$ & $252$ & $56.5973235931712$ \\
$3$ & $8$ & $3.75877048314363$ & $43$ & $306$ & $73.0570345250152$ & $83$ & $586$ & $117.037087585102$ \\
$4$ & $6$ & $2.44948974278318$ & $44$ & $156$ & $35.5537710214148$ & $84$ & $192$ & $34.9899263168499$ \\
$5$ & $16$ & $5.89024980807019$ & $45$ & $120$ & $27.1680532695269$ & $85$ & $448$ & $96.6076285872280$ \\
$6$ & $0$ & $0.000000000000000$ & $46$ & $162$ & $38.2563920108131$ & $86$ & $222$ & $45.5669196048687$ \\
$7$ & $30$ & $10.6212055278435$ & $47$ & $334$ & $76.0785405490705$ & $87$ & $344$ & $63.6184752024529$ \\
$8$ & $24$ & $8.74368977583278$ & $48$ & $96$ & $22.2698637680974$ & $88$ & $192$ & $40.2417450446403$ \\
$9$ & $24$ & $8.46211760746388$ & $49$ & $234$ & $49.5345990909473$ & $89$ & $700$ & $142.765142290188$ \\
$10$ & $24$ & $8.50885617465906$ & $50$ & $132$ & $31.3390062462300$ & $90$ & $96$ & $16.7365006938530$ \\
$11$ & $34$ & $8.06373872165313$ & $51$ & $128$ & $26.8051369832377$ & $91$ & $432$ & $90.9540326194120$ \\
$12$ & $24$ & $6.90183932530257$ & $52$ & $120$ & $28.2286705109218$ & $92$ & $276$ & $55.9590294683383$ \\
$13$ & $60$ & $17.6625147140923$ & $53$ & $304$ & $58.8064777885113$ & $93$ & $384$ & $78.9992542510420$ \\
$14$ & $18$ & $4.51004461341033$ & $54$ & $48$ & $13.3209735097928$ & $94$ & $234$ & $46.1408533394733$ \\
$15$ & $32$ & $10.9523800561763$ & $55$ & $328$ & $72.1117079402388$ & $95$ & $504$ & $101.308022987726$ \\
$16$ & $36$ & $10.5908400750848$ & $56$ & $168$ & $38.8741356168516$ & $96$ & $216$ & $45.3977923880827$ \\
$17$ & $100$ & $28.8451266799323$ & $57$ & $240$ & $54.5982902795300$ & $97$ & $672$ & $121.240085779400$ \\
$18$ & $48$ & $11.6357787307730$ & $58$ & $192$ & $43.9988770726292$ & $98$ & $318$ & $71.4261760057502$ \\
$19$ & $90$ & $22.6680074518870$ & $59$ & $370$ & $86.8139150720411$ & $99$ & $480$ & $105.327246473209$ \\
$20$ & $24$ & $8.62256402186157$ & $60$ & $120$ & $23.3476708044070$ & $100$ & $360$ & $69.3186466692249$ \\
$21$ & $72$ & $21.5754050648131$ & $61$ & $396$ & $88.8088449383063$ & $101$ & $784$ & $163.899681424127$ \\
$22$ & $54$ & $12.8238722222672$ & $62$ & $186$ & $37.1923900968198$ & $102$ & $240$ & $45.3783981921531$ \\
$23$ & $118$ & $31.2713176779699$ & $63$ & $192$ & $39.4977755536762$ & $103$ & $702$ & $141.046737147506$ \\
$24$ & $24$ & $4.68062625874532$ & $64$ & $252$ & $65.9323483074285$ & $104$ & $288$ & $54.0158720618784$ \\
$25$ & $92$ & $23.0405284170969$ & $65$ & $336$ & $78.7891012360799$ & $105$ & $432$ & $80.9717796948430$ \\
$26$ & $84$ & $23.6003748220743$ & $66$ & $168$ & $36.1781508818000$ & $106$ & $360$ & $67.3821144626558$ \\
$27$ & $120$ & $32.1615253767701$ & $67$ & $426$ & $98.8288262449906$ & $107$ & $706$ & $139.080897319654$ \\
$28$ & $60$ & $16.5743373851450$ & $68$ & $240$ & $52.2333340434603$ & $108$ & $264$ & $49.1275927376387$ \\
$29$ & $184$ & $47.9243293316807$ & $69$ & $272$ & $54.9017052415375$ & $109$ & $780$ & $158.620245756362$ \\
$30$ & $72$ & $16.1604866258962$ & $70$ & $192$ & $35.1304032277211$ & $110$ & $144$ & $27.2025670023814$ \\
$31$ & $198$ & $52.2249030090458$ & $71$ & $358$ & $71.4349749764530$ & $111$ & $576$ & $130.288464328520$ \\
$32$ & $84$ & $18.4576382654540$ & $72$ & $168$ & $41.9944276907149$ & $112$ & $408$ & $77.5058407028028$ \\
$33$ & $104$ & $23.1126468409856$ & $73$ & $456$ & $98.5821717496100$ & $113$ & $772$ & $153.976310942955$ \\
$34$ & $84$ & $21.3802759723862$ & $74$ & $252$ & $53.8089795671028$ & $114$ & $336$ & $75.1039529939772$ \\
$35$ & $96$ & $23.2788573806613$ & $75$ & $376$ & $82.6327687606048$ & $115$ & $688$ & $134.930035438987$ \\
$36$ & $72$ & $18.0074606478609$ & $76$ & $276$ & $57.9389308181890$ & $116$ & $408$ & $79.5199162732558$ \\
$37$ & $180$ & $37.9144960695286$ & $77$ & $396$ & $92.2328292717801$ & $117$ & $432$ & $87.0726283552659$ \\
$38$ & $102$ & $25.8385873231311$ & $78$ & $216$ & $42.1228698589956$ & $118$ & $390$ & $65.5233353417242$ \\
$39$ & $144$ & $33.1989215841068$ & $79$ & $534$ & $115.964839305033$ & $119$ & $672$ & $123.549686340431$ \\
$40$ & $72$ & $12.0023106035007$ & $80$ & $216$ & $45.2715362547579$ & $120$ & $168$ & $40.9277382709894$ \\
      \hline
    \end{tabular}
    \caption{Values of the sequences $a_c$ and $\abs{b_c}$ for small values of $c$.}
    \label{t:values}
  \end{table}
\end{center}

\begin{center}
  \begin{figure}
\begin{tabular}{cc}
\includegraphics[scale=0.5]{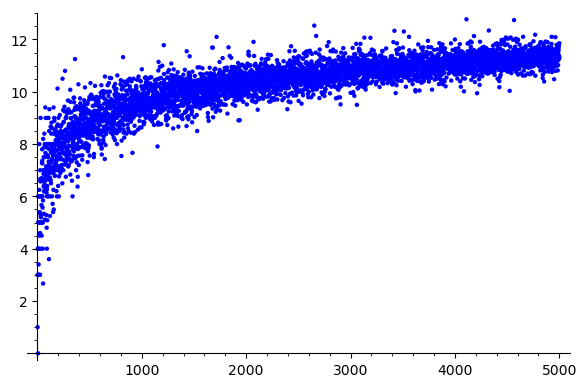}
&\includegraphics[scale=0.5]{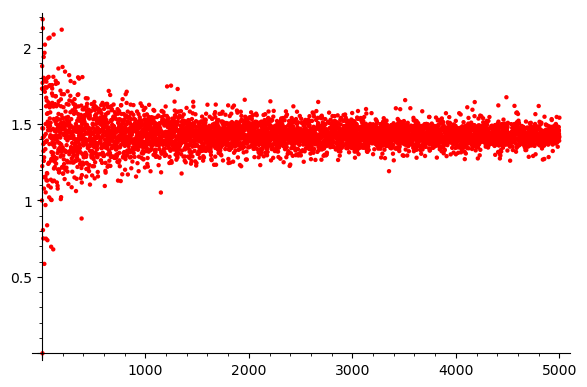}
\end{tabular}
\caption{Values of the sequences $\frac{a_c}{\phi(c)}$ and $\frac{\abs{b_c}}{\phi(c)}$ in blue and red, respectively.}
    \label{f:plots}
\end{figure}
\end{center}

\bibliographystyle{plain}

\end{document}